\documentclass[12pt]{article}

\usepackage{amssymb,amsmath,amsfonts,amssymb}
\usepackage{graphics,graphicx,color,hhline}
\usepackage{bbm} 
\usepackage{authblk}
\usepackage{euscript}
 \usepackage{mathtools}

\def\@abssec#1{\vspace{.05in}\footnotesize \parindent .2in 
{\bf #1. }\ignorespaces} 
\graphicspath{{/EPSF/}{Figures/}}   
\newtheorem{theorem}{Theorem}[section]
\newtheorem{lemma}[theorem]{Lemma}
\newtheorem{proposition}[theorem]{Proposition}
  
\newtheorem{definition}[theorem]{Definition}
\newtheorem{remark}[theorem]{Remark}

\def\ds{\displaystyle}
\def \Rm {\mathbb R}
\def \Nm {\mathbb N}
\def \NN {\mathbb N}
\def \HH {\mathbb H}
\def \Cm {\mathbb C}

\def \AA {\mathbb A}
\def \BB {\mathbb B}

\def \WW {\mathbb W}

\newcommand{\eps}{\varepsilon}

\newcommand{\be}{\begin{equation}}
\newcommand{\ee}{\end{equation}}
\newcommand{\bea}{\begin{eqnarray}}
\newcommand{\eea}{\end{eqnarray}}
\newcommand{\bee}{\begin{eqnarray*}}
\newcommand{\eee}{\end{eqnarray*}}

\def\fref#1{{\rm (\ref{#1})}}

\newcommand{\calL}{\mathcal L}
\newcommand{\calF}{\mathcal F}
\newcommand{\calM}{\mathcal M}
\newcommand{\frakF}{\mathfrak F}

\newcommand{\calE}{\mathcal E}

\newcommand{\calN}{\mathcal N}

\newcommand{\calS}{\mathcal S}
\newcommand{\calJ}{\mathcal J}
\newcommand{\calA}{\mathcal A}

\newcommand{\cout}[1]{}

\newcommand{\frakh}{\mathfrak{h}}

\DeclareMathOperator{\Tr}{Tr}
\DeclarePairedDelimiter\bra{\langle}{\rvert}
\DeclarePairedDelimiter\ket{\lvert}{\rangle}

\begin{document}
 \title{Entropy minimization for many-body quantum systems}
\author{Romain  Duboscq \footnote{Romain.Duboscq@math.univ-tlse.fr}}
 \affil{Institut de Math\'ematiques de Toulouse ; UMR5219\\Universit\'e de Toulouse ; CNRS\\INSA, F-31077 Toulouse, France}
 \author{Olivier Pinaud \footnote{pinaud@math.colostate.edu}}
 \affil{Department of Mathematics, Colorado State University\\ Fort Collins CO, 80523}
 \maketitle
 \begin{abstract}The problem considered here is motivated by a work by B. Nachtergaele and H.T. Yau where the Euler equations of fluid dynamics are derived from many-body quantum mechanics, see \cite{nachter}. A crucial concept in their work is that of local quantum Gibbs states, which are quantum statistical equilibria  with prescribed particle, current, and energy densities at each point of space (here $\Rm^d$, $d \geq 1$). They assume that such local Gibbs states exist, and show that if the quantum system is initially in a local Gibbs state, then the system stays, in an appropriate asymptotic limit, in a Gibbs state with particle, current, and energy densities now solutions to the Euler equations. Our main contribution in this work is to prove that such local quantum Gibbs states can be constructed from prescribed densities under mild hypotheses, in both the fermionic and bosonic cases. The problem consists in minimizing the von Neumann entropy in the quantum grand canonical picture under constraints of local particle, current, and energy densities. The main mathematical difficulty is the lack of compactness of the minimizing sequences to pass to the limit in the constraints. The issue is solved by defining auxiliary constrained optimization problems, and by using some monotonicity properties of equilibrium entropies.
 \end{abstract}
\section{Introduction}
This work is concerned with the minimization of the von Neumann entropy
$$
S(\varrho)= \Tr (\varrho \log \varrho),
$$
where $\varrho$ is a nonnegative, trace class operator with trace one, on some infinite-dimensional Hilbert space (we will refer to such operator in the sequel as a \textit{state}, which are self-adjoint since nonnegative). Note the sign change in the entropy compared to the standard definition in the physics literature. The problem is motivated by the work of Nachtergaele and Yau addressed in \cite{nachter}, where they derive the Euler equations of fluid mechanics from quantum dynamics. More precisely, they consider the entropy minimization problem in the context of many-body quantum mechanics where the underlying Hilbert space is the \textit{Fermionic Fock space}. The latter is defined as follows: let $\mathfrak{h}=L^2(\Rm^d)$ for $d \geq 1$; the Fermionic Fock space $\mathfrak{F}_f$ is the direct sum
\begin{equation*}
\mathfrak{F}_f : = \bigoplus_{n = 0}^{+\infty} \mathfrak{F}^{(n)}_f,
\end{equation*}
where $\mathfrak{F}_f^{(n)} : = \mathfrak{h}^{\wedge_s n}$ is the $n$-fold antisymmetric tensor product of $\mathfrak{h}$, with the convention $\mathfrak{h}^{\wedge_s 0} = \mathbb{C}$. We have
$$
\mathfrak{F}_f^{(n)}=L^2_a\big((\Rm^d)^n\big),
$$
where $L^2_a\big((\Rm^d)^n\big)$ is the space of antisymmetric square integrable complex functions on $(\Rm^d)^n$, that is, for $x_\ell \in \Rm^d$ , $\ell=1,\cdots,n$, $1 \leq, i,j \leq n$, $$f(x_1,\cdots,x_i,\cdots,x_j,\cdots,x_n)=-f(x_1,\cdots,x_j,\cdots, x_i,\cdots,x_n)$$
when $f \in L^2_a((\Rm^d)^n)$. This setting is usually referred to as the grand canonical picture since the system is not fixed to a particular subspace with $n$ particles.

Nachtergaele and Yau consider the minimization of $S(\varrho)$ over states $\varrho$ with \textit{prescribed particle density, current and energy densities at any given point $x \in \Rm^d$}. This results in an infinite dimensional constrained optimization problem, whose main mathematical difficulty is to handle the local nature of the constraints. The solution can be seen as the quantum many-body equivalent of the classical Maxwellian obtained by minimizing the Boltzmann entropy under local constraints of density, current and energy.

Using the formalism of the second quantization, the local constraints can be defined as follows: let $\{e_i\}_{i \in \NN}$ be an orthonormal basis of $\frakh$, let $x \in \Rm^d$, and consider (formally) the following family of operators parametrized by $x$,
$$
a_x=\sum_{i \in \NN} e_i^*(x) a(e_i),
$$
where $a(\cdot)$ is the annihilation operator and $e_i^*$ is the complex conjugate of $e_i$. The adjoint of $a_x$ in $\frakF_f$  is denoted by $a^*_x$. We do not give the explicit definition of $a(\cdot)$ since it will not be needed in the sequel, and we point the reader to \cite{arai} for instance for more details.  For a state $\varrho$ and $\Tr(\cdot)$ the trace in $\mathfrak{F}_f$, we introduce the following functions of $x$,

\be \label{consint}
\left\{
\begin{array}{ll}
\ds n[\varrho](x):=\Tr (a^*_x a_x\, \varrho), \qquad &\textrm{local density}\\
\ds u[\varrho](x): = \Im \Tr (a^*_x \nabla a_x\, \varrho),\qquad &\textrm{local current}\\ 
\ds k[\varrho](x):=\Tr (\nabla a^*_x \cdot \nabla a_x\, \varrho)\qquad &\textrm{local kinetic energy},
\end{array}
\right.
\ee
where the gradient $\nabla$ is taken with respect to the variable $x$, and $\Im$ denotes imaginary part. We defined above the local kinetic energy instead of the total energy which includes (two-body) interactions between the particles. The latter will be defined further on. The formulas in \fref{consint} are similar in structure to the usual definitions of the density, current and kinetic energy for the one-particle setting. We will introduce in the sequel equivalent definitions based on one-particle density matrices that are more amenable to mathematical rigor.

Introducing the potential energy $e_P[\varrho]= V n[\varrho]$ for some potential $V$, and the total energy $e[\varrho]=k[\varrho]+e_P[\varrho]+e_I[\varrho]$ for some two-body interaction term $e_I[\varrho]$, Nachtergaele and Yau assume that the minimization problem with constraints on $n[\varrho]$, $u[\varrho]$ and $e[\varrho]$ admits a unique solution, referred to as a \textit{local Gibbs state} (or more accurately they give an informal expression of the statistical equilibrium that is the solution to the constrained minimization problem). Then they prove that a state $\varrho_t$ solution to the quantum Liouville equation
%\be \label{liou}
$$
i \partial_t \varrho_t=[H,\varrho_t], \qquad [H,\varrho_t]=H\varrho_t-\varrho_t H, \qquad H=-\Delta+V+W,
$$
%\ee
with a local Gibbs state with constraints $\{n_0,u_0,e_0\}$ as initial condition, converges, in an appropriate limit that we do not detail here, to a local Gibbs state with constraints $\{n_0(t),u_0(t),e_0(t)\}$. These latter constraints are solutions to the Euler equations with initial condition $\{n_0,u_0,e_0\}$. Above, $W$ is a two-body interaction potential used to define $e_I[\varrho]$.

In the one-body case, a similar constrained entropy minimization problem is central to the work of Degond and Ringhofer in their derivation of quantum fluid models from quantum dynamics, see \cite{DR}.

Under appropriate conditions on $\{n_0,u_0,e_0\}$, our main result in this work is justify rigorously the introduction of these local Gibbs states, and therefore to prove that indeed the constrained minimization problem admits a unique solution, \textit{for both the fermionic and bosonic cases}. We addressed in \cite{DP-CMP} the one-body problem for various quantum entropies, and the key difficulty is the lack of compactness required to handle the local energy constraint. The many-body case treated here introduce new difficulties, in particular the fact that there is now also a lack of compactness to treat the local density constraint, and as a consequence the current constraint. In the one-body setting, there is no such issue with the density since sequences of states with bounded energy have automatically sufficient compactness to pass to the limit in the density constraint, while this is not true in the many-body case. This latter fact is related to the convergence of one-body density matrices that will be defined further.

The main idea to go around the issue is to define two auxiliary optimization problems with \textit{global} constraints, and to prove the monotonicity of the entropy of the corresponding minimizers with respect to these global constraints. Along with classical compactness theorems for trace class operators, this allows us to prove that, while arbitrary sequences of states with bounded energy do not have sufficient compactness, the minimizing sequences of the entropy converge in a sufficiently strong sense that allows us to pass to the limit in the local constraints. The fermionic and bosonic cases are treated in the same fashion with essentially identical proofs.

The article is structured as follows: in Section \ref{prelim}, we introduce some background on second quantization; we next state our main result in Section \ref{mainR}, and prove the main theorem in Section \ref{proofT}. Finally, some standard technical results are given in an appendix 

\paragraph{Acknowledgment.} OP's work is supported by NSF CAREER Grant DMS-1452349 and NSF grant DMS-2006416. 

\section{Preliminaries} \label{prelim}
We introduce in this section some background that will be used throughout the paper.

\paragraph{Second quantization formalism.}  We have already defined the fermionic Fock space in the introduction, and define now the bosonic version, denoted by $\mathfrak{F}_b$. It is given by  the direct sum
\begin{equation*}
\mathfrak{F}_b : = \bigoplus_{n = 0}^{+\infty} \mathfrak{F}^{(n)}_b
\end{equation*}
where $\mathfrak{F}_b^{(n)} : = \mathfrak{h}^{\otimes_s n}$ is the $n$-fold symmetric tensor product of $\mathfrak{h}$ and $\mathfrak{h}^{\otimes_s 0} := \mathbb{C}$. We have
$$
\mathfrak{F}_b^{(n)}=L^2_s\big((\Rm^d)^n\big),
$$
where $L^2_s\big((\Rm^d)^n\big)$ is the space of symmetric square integrable complex functions on $(\Rm^d)^n$, that is, for $x_\ell \in \Rm^d$ , $\ell=1,\cdots,n$, $1 \leq, i,j \leq n$,
$$f(x_1,\cdots,x_i,\cdots,x_j,\cdots,x_n)=f(x_1,\cdots,x_j,\cdots,x_i,\cdots,x_n)$$
 when $f \in L^2_s((\Rm^d)^n)$. 

We denote by $\mathfrak{F}_{b/f}$ either the bosonic or fermionic Fock space, and represent an element $\psi$ of $\mathfrak{F}_{b/f}$ by the sequence $\psi=\{\psi^{(n)}\}_{n \in \NN}$, where $\psi^{(n)} \in \mathfrak{F}_{b/f}^{(n)}$. The spaces $\mathfrak{F}_{b/f}$ are Hilbert spaces when equipped with the norm
$$
\| \psi\| = \left( \sum_{n \in \NN} \| \psi^{(n)}\|^2_n\right)^{1/2}, \qquad \| \psi^{(n)}\|_n:=\| \psi^{(n)}\|_{L^2_{s/a}((\Rm^{d})^n)}.
$$
We use the same notation for the norms in  $\mathfrak{F}_{s}$ and $\mathfrak{F}_{a}$ since there will be no possible confusion in the sequel. The inner products associated with $\|\cdot \|$ and $\| \cdot \|_n$ are denoted by $(\cdot,\cdot)$ and $(\cdot,\cdot)_n$.

We denote by $\calJ_1:=\calJ_1(\mathfrak{F}_{b/f})$ the space of trace class operators on $\mathfrak{F}_{b/f}$. The trace with respect to $\calJ_1$ is denoted simply by $\Tr(\cdot)$, while the trace with respect to $\calJ_1(E)$ for $E$ a Hilbert space is denoted by $\Tr_E(\cdot)$. The space of bounded operators on $E$ is denoted $\calL(E)$.

We will refer to a ``state'', as a nonnegative, trace class operator on $\mathfrak{F}_{b/f}$ with trace equal to one. The set of states is denoted by $\calS$, i.e.
$$
\calS=\left\{\varrho \in \calJ_1: \; \varrho \geq 0,\quad \Tr(\varrho)=1 \right\}.
$$
  \begin{definition}(Second quantization) 
%For any densely defined closable operator $H$ on $\mathfrak{h}$ and any $n\in\mathbb{N}$, we set
%\begin{equation*}
%H_0^{(n)} : = \sum_{k = 0}^n  \mathrm{Id}^{\otimes k}\otimes H \otimes \mathrm{Id}^{\otimes n-k-1}
%\end{equation*}
%defined on $\bigotimes^n D(H)$. This operator is closable and we denote $H^{(n)} = \overline{H_0^{(n)}}$ its closure. By setting $T^{(0)} : = 0$ acting on $\mathbb{C}$, we define the second quantization of $H$ as
%\begin{equation*}
%d\Gamma(H) = \oplus_{n = 0}^{+\infty} H^{(n)}.
%\end{equation*}
%In this context, $H$ is called a one-particle operator and can also be denoted $H^{(1)}$.
    Let $A$ be an operator acting on $\mathfrak{F}^{(k)}_{b/f}$, $k \geq 1$.  Its second quantization, denoted $\AA$, is defined by
    $$
      \mathbb{A} := \underbrace{0\oplus\cdots\oplus 0}_{k\;\mathrm{times}}  \oplus \bigoplus_{n = k}^{+\infty}\; \sum_{1\leq i_1<\cdots< i_k \leq n} (A)_{i_1,\cdots,i_k}
      $$
where $(A)_{i_1,\cdots,i_k}$ is the operator $A$ acting on the variables labeled $i_1,\cdots,i_k$ in $\mathfrak{F}_{b/f}^{(k)}$ and leaving the other variables invariant.

% \begin{equation*}
% \mathbb{A} := 0 \oplus \bigoplus_{n = 1}^{+\infty} \mathbb{A}_{(n)},
% \end{equation*}
% with
% \begin{equation*}
% \mathbb{A}_{(n)} = \sum_{1\leq i \leq n} A_{(i)},
% \end{equation*}
% where $A_{(i)}$ is the operator $A$ acting on the variables $x_i \in \Rm^d$ in $\mathfrak{F}_{f/b}^{(n)}$.
\end{definition}
It is customary to denote $d\Gamma(A)=\AA$. The second quantization of the identity on $\frakh$ is the number operator
\begin{equation*}
\mathcal{N} : = d\Gamma(\mathrm{Id}_{\mathfrak{h}}).
\end{equation*}
The identity on $\mathfrak{F}_{b/f}$ is simply denoted by $\mathrm{Id}$. The operator $\calN$ is self-adjoint on $\mathfrak{F}_{b/f}$ when equipped with the domain
$$
D(\calN)=\left\{ \psi=\{\psi^{(n)}\}_{n \in \NN} \in \mathfrak{F}_{b/f}: \; \sum_{n \in \NN} n^2 \|\psi^{(n)}\|^2_n < \infty\right\}.
$$
We consider states that are not necessarily diagonal on $\mathfrak{F}_{b/f}$, and therefore that do not commute in general with $\calN$. In particular, $\calN \varrho$ is not necessarily positive when $\varrho \geq 0$, and this leads us to introduce $\calN^{1/2}$, which is self-adjoint on $\mathfrak{F}_{b/f}$ with domain
$$
D(\calN^{1/2})=\left\{ \psi=\{\psi^{(n)}\}_{n \in \NN} \in \mathfrak{F}_{b/f}: \; \sum_{n \in \NN} n \|\psi^{(n)}\|^2_n < \infty\right\}.
$$
We then denote by $\calS_0$ the set of states $\varrho$ with finite average particle number, that is such that
$$
\Tr \left(\overline{\calN^{1/2} \varrho\, \calN^{1/2}}\right)<\infty,
$$
where $\overline{\AA}$ denotes the extension of an operator $\AA$ to $\mathfrak{F}_{b/f}$. We will drop the extension sign in the sequel for simplicity. 

\begin{definition}(1-particle density matrix) \label{onebody}
 Let $A$ be a bounded operator on $\frakh$ and consider its second quantization $\AA$. For a state  $\varrho \in \calS_0$, the 1-particle density matrix $\varrho^{(1)}$ is the unique nonnegative operator in $\calJ_1(\frakh)$ such that
\begin{equation} \label{def1}
\mathrm{Tr}_{\frakh}\left(A \varrho^{(1)}\right) = \mathrm{Tr} \left (\AA \varrho \right).
\end{equation}
\end{definition}

The fact that $\varrho^{(1)}$ is well-defined is classical and is established in Appendix for the sake of completeness. Note that since $\AA$ is not bounded in $\frakF_{b/f}$, relation \fref{def1} has actually to be understood as
$$
\mathrm{Tr}_{\frakh}\left(A \varrho^{(1)}\right) = \mathrm{Tr} \left (\BB \calN^{1/2} \varrho\,\calN^{1/2} \right),
$$
where
\be \label{defB}
\BB:=  0 \oplus \bigoplus_{n = 1}^{+\infty} n^{-1}\mathbb{A}_{(n)},
\ee
and then belongs to $\calL(\mathfrak{F}_{b/f})$ when $A \in \calL(\frakh)$. In \fref{defB}, $\mathbb{A}_{(n)}$ is the component of $\AA$ on the sector $\frakF_{b/f}^{(n)}$. Note that by setting $A=\mathrm{Id}_{\mathfrak{h}}$, we have the relation
$$
\mathrm{Tr}_{\frakh}\left(\varrho^{(1)}\right) = \mathrm{Tr} \left( \calN^{1/2} \varrho\ \calN^{1/2} \right).
$$
We will need as well the 2-particle density matrix for the definition the interaction potential. It is justified in the same manner as Definition \ref{onebody}.

\begin{definition}(2-particle density matrix)  Let $A$ be a bounded operator on $\frakF^{(2)}_{b/f}$ and consider its second quantization $\AA$. For a state  $\varrho \in \calS$ such that $\Tr \left(\calN \varrho\, \calN\right)$ is finite, the 2-particle density matrix $\varrho^{(2)}$ is the unique nonnegative operator in $\calJ_1(\frakF^{(2)}_{b/f})$ such that
%\begin{equation} \label{def2}
  $$
\mathrm{Tr}_{\frakF^{(2)}_{b/f}}\left(A \varrho^{(2)}\right) = \mathrm{Tr} \left (\AA \varrho \right).
$$
%\end{equation}
\end{definition}

We define now the local constraints, first the density, current, and kinetic energy.

\paragraph{Local density, current, and kinetic energy constraints.}
Consider a state $\varrho \in \calS_0$ and its associated 1-particle density matrix $\varrho^{(1)} \in \calJ_1(\frakh)$. The local (1-particle) density $n[\varrho]$ of $\varrho$ is defined by duality by, for any $\varphi\in C^{\infty}_0(\mathbb{R}^d)$:
\begin{equation*}
\int_{\mathbb{R}^d} n[\varrho](x) \varphi(x) dx = \mathrm{Tr}_{\mathfrak{h}}\left(\varrho^{(1)} \varphi\right),
\end{equation*}
where we identify $\varphi$ and its associated multiplication operator.

Let $h_0=-\Delta$, equipped with domain $H^2(\Rm^d)$, and for $ \HH_0=d\Gamma(h_0)$, let $\calE_0$ be the following set:
$$
\calE_0=\left\{ \varrho \in \calS_0: \; \Tr\left( \HH_0^{1/2} \varrho \,\HH_0^{1/2} \right)<\infty \right\}.
$$
$\calE_0$ is the set of states with finite particle number and finite kinetic energy. We will need the following lemma in order to define the current and energy contraints. The straightforward proof is given in the Appendix for convenience of the reader. 

\begin{lemma} \label{enrho} Let $\varrho \in \calE_0$. Then $\varrho^{(1)}$ verifies $\mathrm{Tr}_\frakh \left( h_0^{1/2} \varrho^{(1)}\ h_0^{1/2} \right)<\infty$ and

  \be \label{eqkin} \mathrm{Tr}_\frakh \left( h_0^{1/2} \varrho^{(1)} h_0^{1/2} \right)= \Tr\left( \HH_0^{1/2} \varrho \,\HH_0^{1/2} \right).\ee
\end{lemma}

The current $u[\varrho]$ can be now defined by, for any $\Phi \in (C^{\infty}_0(\mathbb{R}^d))^d$,
\begin{equation*}
\int_{\mathbb{R}^d} u[\varrho](x) \cdot \Phi(x) dx = - i \mathrm{Tr}_{\mathfrak{h}}\left(\Phi \cdot \nabla \varrho^{(1)} +\frac{1}{2} \nabla \cdot \Phi \right), 
\end{equation*}
and  the kinetic energy $k[\varrho]$ by 
\begin{equation*}
\int_{\mathbb{R}^d} k[\varrho](x) \varphi(x) dx = - \mathrm{Tr}_{\mathfrak{h}}\left(\nabla \cdot (\varphi \nabla)\varrho^{(1)} \right).
\end{equation*}
 
Note that $u[\varrho]$ and $k[\varrho]$ are well-defined since Lemma \ref{enrho} implies that $\nabla \varrho^{(1)}$ and $\nabla \varrho^{(1)} \nabla$ are trace class. Formal calculations also show that $n[\varrho]$, $u[\varrho]$ and $k[\varrho]$ agree with the definitions given in \fref{consint} in the introduction based on the annihilation operator. Moreover, if $\{\mu_p\}_{p \in \Nm}$ and $\{\varphi_p\}_{p \in \Nm}$ denote the eigenvalues and eigenfunctions of $\varrho^{(1)}$, we have the familiar relations
%\be \label{consint2}
$$
\left\{
\begin{array}{ll}
\ds n[\varrho]=\sum_{p \in \Nm} \mu_p |\varphi_p|^2, \qquad &\textrm{local density}\\
\ds u[\varrho] = \sum_{p\in\mathbb{N}} \mu_p \Im\left( \varphi_p^* \nabla \varphi_p\right),\qquad &\textrm{local current}\\ 
\ds k[\varrho]=\sum_{p \in \Nm} \mu_p |\nabla \varphi_p|^2, \qquad &\textrm{local kinetic energy}.
\end{array}
\right.
$$
% \ee

The functions $n[\varrho]$, $u[\varrho]$, and $k[\varrho]$ are all in $L^1(\Rm^d)$ when $\varrho \in \calE_0$, and the series above converge in $L^1(\Rm^d)$. We have moreover the relations
$$
\|n[\varrho]\|_{L^1}=\Tr\left( \calN^{1/2} \varrho \,\calN^{1/2} \right), \qquad \|k[\varrho]\|_{L^1}=\Tr\left( \HH_0^{1/2} \varrho \,\HH_0^{1/2} \right).
$$

\paragraph{Definition of local total energy.} We define now the local potential and local interaction energy constraints. For this, let $v=v_+-v_-$, $v_\pm \geq 0$, and $w$ even, all real-valued, such that
$$
v_+ \in L^1_{\rm{loc}}(\Rm^d), \qquad v_-, w \in L^p(\Rm^d)+L^\infty(\Rm^d),
$$
with $p=1$ when $d=1$, $p>1$ when $d=2$, and $p=d/2$ when $d \geq 3$. We suppose that $w$ is classically stable of the second kind, that is, there exists a constant $C_0>0$ such that 
\be \label{class}
 \forall n \geq 2, \qquad \sum_{1 \leq i<j \leq n} w(x_i-x_j) \geq -C_0 n, \qquad \textrm{a.e. on } (\Rm^d)^n.
\ee
An example of such $w$ is the standard Coulomb potential $w(x)=|x|^{-1}$ when $d=3$.

For a state $\varrho$ with $\Tr \left(\calN \varrho\, \calN\right)<\infty$, the local interaction energy is formally defined by
$$
\widetilde e_{I}[\varrho](x)=\int_{\Rm^d} n[\varrho^{(2)}](x,y) w(x-y) dy, 
$$
where $n[\varrho^{(2)}](x,y)$ is the local density associated with the 2-particle density matrix $\varrho^{(2)}$ of $\varrho$. It is defined by duality by
\begin{equation*}
\int_{\mathbb{R}^d \times \Rm^d} n[\varrho^{(2)}](x,y) \varphi(x,y) dx dy = \mathrm{Tr}_{\mathfrak{F}_{b/f}^{(2)}}\left(\varrho^{(2)} \varphi\right),
\end{equation*}
for any symmetric test function $\varphi \in C^\infty_0(\Rm^d \times \Rm^d)$. The condition $\Tr \left(\calN \varrho\, \calN\right)<\infty$ is not natural in the minimization problem since the total energy constraint only involves $\widetilde{e}_I[\varrho]$ and not $\calN \varrho\, \calN$, and  it is therefore not clear how to define $n[\varrho^{(2)}]$ rigorously if $\Tr \left(\calN \varrho\, \calN\right)$ is not finite. We then introduce a modified interaction energy $e_I[\varrho]$ as follows. Let $r_0>C_0$ (for the $C_0$ defined in \fref{class}) and let
$$
\WW=  0 \oplus 0 \oplus \bigoplus_{n = 2}^{+\infty} \left(\sum_{1 \leq i<j \leq n} w(x_i-x_j) \right).
$$

The operator $\WW+r_0 \calN$ is strictly positive according to \fref{class}. Consider a state $\varrho$ such that
$$\Tr \left((\WW+r_0 \calN )^{1/2}\varrho\, (\WW+r_0 \calN )^{1/2}\right)<\infty,$$
and let
$$
\BB:=0 \oplus r_0^{1/2}  \oplus \bigoplus_{n = 2}^{+\infty}\left(\sum_{1 \leq i<j \leq n} w(x_i-x_j) + n r_0\right)^{1/2}  n^{-1}.
$$
The operator $\sigma_\varrho=\BB \varrho\, \BB$ verifies as a consequence $\Tr \left(\calN \sigma_\varrho\, \calN\right)<\infty$, and therefore has a  2-particle density matrix $\sigma^{(2)}_\varrho $. We then define
$$
e_I[\varrho](x):=\int_{\Rm^d} n[\sigma^{(2)}_\varrho](x,y)dy \geq 0 \qquad a.e.
$$
Note that $e_I[\varrho]$ is integrable since $n[\sigma^{(2)}] \in L^1(\Rm^d \times \Rm^d)$, and that
$$
\Tr \left((\WW+r_0 \calN )^{1/2}\varrho\, (\WW+r_0 \calN )^{1/2}\right)=\|e_I[\varrho]\|_{L^1}.
$$
When $\Tr \left(\calN \varrho\, \calN\right)$ is finite, the two definitions of the interaction energy given above are equivalent for the minimization problem since 
$$
e_I[\varrho]=\widetilde e_{I}[\varrho]+r_0 n[\varrho],
$$
and prescribing both $e_I[\varrho]$ and $n[\varrho]$ is then equivalent to prescribing both $\widetilde{e}_I[\varrho]$ and $n[\varrho]$.

Regarding the potential energy, the density $n[\varrho]$ must have sufficient decay at the infinity for the entropy of a state $\varrho$ to be bounded below. We then introduce a nonnegative confining potential $v_c \in L^1_{\rm{loc}}(\Rm^d)$ with $v_c \to +\infty$ as $|x| \to +\infty$ such that
\be \label{GT}
\forall t>0, \qquad  \int_{\Rm^d} e^{- t v_c(x)} dx < \infty,
\ee
and suppose that $v_c n[\varrho] \in L^1(\Rm^d)$. With $h_c=h_0+v_c$, defined in the sense of quadratic forms and self-adjoint on an appropriate domain, the condition \fref{GT} ensures by the Golden-Thompson inequality that the operator $e^{-t h_c}$ is trace class for all $t>0$.

Let finally $V=v+v_c$. The local potential energy is defined by
$$
e_P[\varrho]= V n[\varrho].
$$
Again, for the minimization problem, prescribing both $V n [\varrho]$ and $n[\varrho]$ is equivalent to prescribing both $v n [\varrho]$ and $n[\varrho]$, and therefore the introduction of $v_c$ in the constraint does not change the minimizer. The local total energy of a state $\varrho$ is then
$$
e[\varrho]=k[\varrho]+e_P[\varrho]+e_I[\varrho].
$$

\paragraph{The energy space.} For $n\geq 2$, let the symmetric $n$-body operator
$$
H_n=\sum_{j=1}^n \left(-\Delta_{x_j} + V(x_j)+r_0 \right)+\sum_{1 \leq i<j \leq n} w(x_i-x_j),
$$
 with for $n=1$,
  $$
  H_1=-\Delta+V+r_0.
  $$
The regularity assumptions on $v_-$ and $w$ imply that the $n$-body potential in $H_n$
$$
\sum_{j=1}^n v_-(x_j)+\sum_{1 \leq i<j \leq n} w(x_i-x_j)
$$
is infinitesimally $(-\Delta)$-form bounded as an adaptation of Kato's theorem, see e.g. \cite{RS-80-2} in the case $d=3$. A first consequence of this is that there exists a constant $\gamma \in (0,1)$ independent of $n$, and that $r_0$ can be chosen sufficiently large, such that
\be \label{Hbelow}
\gamma\, H_n^c + \gamma n\leq H_n, \qquad H_n^c=\sum_{j=1}^n \left(-\Delta_{x_j} + v_c(x_j) \right),
  \ee
  in the sense of operators. Second, $H_n$ is associated with a quadratic form closed on the space of functions $\psi \in \frakF_{b/f}^{(n)}$ such that $\psi \in H^1((\Rm^d)^n)$ and
  $$
  \int_{(\Rm^d)^n} (v_+(x_1)+v_c(x_1))|\psi(x_1,\cdots,x_n)|^2 dx_1 \cdots dx_n < \infty.
  $$

  With an abuse of notation, we will also denote  by $H_n$ the self-adjoint realization with domain $D(H_n)$ of the quadratic form. In the same way, $H_1$ is the self-adjoint realization of $-\Delta+V+r_0$ defined in the sense of quadratic forms.
  
  Let $\HH$ be the second quantization of $H_n$, that is
  $$\HH=d\Gamma(h_0+V+r_0)+\WW.$$
  It is self-adjoint with domain
  $$
  D(\HH)=\Cm \oplus D(h_0+V)\oplus \bigoplus_{n = 2}^{+\infty} D(H_n),
  $$
   see \cite[Theorem 4.2]{arai}. The energy space that we will use in the minimization is finally the following:
  $$
\calE=\left\{ \varrho \in \calS: \; \Tr\left( \HH^{1/2} \varrho \,\HH^{1/2} \right)<\infty \right\}.
$$
Note that
\be \label{eqeneg}
\Tr\left( \HH^{1/2} \varrho \,\HH^{1/2} \right)=\|e[\varrho]\|_{L^1},
\ee
and that condition \fref{Hbelow} yields
\be \label{below2}
\gamma \Tr\left( d\Gamma(h_c)^{1/2} \varrho \,d\Gamma(h_c)^{1/2} \right) + \gamma \Tr\left( \calN^{1/2} \varrho \, \calN^{1/2} \right)
\leq \Tr\left( \HH^{1/2} \varrho \,\HH^{1/2} \right).
\ee
This implies in particular that $\calE \subset \calE_0$.\\

We are now in position to state our main result.

\section{Main result} \label{mainR}

The  entropy of a state $\varrho \in \calS$ is defined by
$$
S(\varrho)= \Tr (\varrho \log \varrho)= \sum_{i \in \NN} \rho_i \log \rho_i,
$$
for $\{\rho_i\}_{i \in \NN}$ the eigenvalues of $\varrho$ (counted with multiplicity and forming a nonincreasing sequence; if $\varrho$ has a finite rank, then $\rho_i=0$ when $i \geq N$ for some $N$). Note that $S$ is always well-defined in $[-\infty,0]$ since $0 \leq \rho_i \leq 1$ as $\Tr(\varrho)=1$.

The set of admissible local constraints (we will sometimes refer to these as ``moments'' in the sequel) is defined by
\begin{align*}
\calM = \Big\{(n,u,e)\in L_+^1(\mathbb{R}^d)\times (L^1(\mathbb{R}^d))^d\times L_+^1(\mathbb{R}^d)\hspace{5cm}
\\ \qquad  \qquad\textrm{such that }\quad (n,u,e) = (n[\varrho],u[\varrho],e[\varrho])\quad\textrm{for at least one }\varrho\in\calE \Big\}.
\end{align*}
Above, $L^1_+(\mathbb{R}^d) = \{ \varphi\in L^1(\mathbb{R}^d):\; \varphi\geq 0\; a.e.\}$. In other terms, $\calM$ consists of the set of functions $(n,u,e)$ that are the local density, current and total energy of at least one state with finite energy. It is not difficult to construct admissible constraints, for instance by taking moments of the Gibbs state
$$
\frac{e^{-\HH}}{ \Tr (e^{-\HH})}.
$$
To the best of our knowledge, the characterization of $\calM$ remains to be done. 

For $(n_0,u_0,e_0) \in \calM$, the feasible set is then given by
\begin{equation*}
\calA(n_0,u_0,e_0) = \Big\{\varrho\in\calE:\; n[\varrho] = n_0,\; u[\varrho] = u_0\;\textrm{and}\; e[\varrho] = e_0 \Big\}.
\end{equation*}
The set $\calA(n_0,u_0,e_0)$ is not empty by construction since  $(n_0,u_0,e_0)$ is admissible. \\

Our main result is the next theorem.

\begin{theorem}\label{thm:Main} Let $(n_0,u_0,e_0) \in \calM$, with $n_0 v_c \in L^1(\Rm^d)$. Then, the minimization problem
  $$
  \inf_{\calA(n_0,u_0,e_0)} S
  $$
  admits a unique solution.
\end{theorem}

We expect the minimizer $\varrho^\star$ to be a local Gibbs state with Hamiltonian $\HH^\star$, for $\HH^\star$ the second quantization of some two-body interaction Hamiltonian involving the Lagrange multipliers (which are functions here) associated with the constraints. While a formal derivation using standard calculus of variations techniques is quite straightforward (this is actually the formal expression given in the work of Nachtergaele and Yau), a rigorous derivation appears to be quite difficult. It was achieved in the one-particle situation in \cite{MP-JSP, DP-CVPDE, DP-JFA} in various settings.

% We remark that each constraint is linear with respect to $\rho$ and, furthermore, we have, for any $\rho\in\mathcal{E}_0$,
% \begin{equation}\label{eq:bndcurr}
% \left|u[\rho]\right| \leq \left(\sum_{p \in \Nm} \mu_p |\varphi_p|^2\right)^{1/2} \left( \sum_{p \in \Nm} \mu_p |\nabla \varphi_p|^2\right)^{1/2} = n[\rho]^{1/2}k[\rho]^{1/2}.
% \end{equation}

\paragraph{Outline of the proof.} It will be shown in Section \ref{secpropS} that the entropy of states with fixed total energy is bounded below, and that the entropy is lower semi-continuous on the energy space $\calE$. The proof of the theorem therefore hinges upon showing that minimizing sequences satisfy the constraints in the limit. Standard weak-$*$ compactness theorems in the space of trace class operators show that minimizing sequences converge in some weak sense to an operator $\varrho^\star$ with finite energy, and the convergence is sufficiently strong to obtain that $\varrho^\star$ is a state, i.e. that $\Tr(\varrho^\star)=1$. Lower semi-continuity of the entropy yields moreover
 \be \label{ineqS}
S(\varrho^\star) \leq   \inf_{\calA(n_0,u_0,e_0)} S.
\ee
It is not possible to identify at that stage the local moments of $\varrho^\star$, and as a consequence to show that $\varrho^\star$ belongs to the feasible set $\calA(n_0,u_0,e_0)$. The core of the proof consists then in showing that the moments of the minimizing sequences converge in a strong sense, allowing us to obtain that
\be \label{egalconst}
n[\varrho^\star] = n_0,\qquad u[\varrho^\star]=u_0, \qquad e[\varrho^\star] = e_0,\ee
which, together with \fref{ineqS} and the strict convexity of $S$, proves Theorem \ref{thm:Main}.

Our strategy to recover strong convergence follows the general method we introduced in \cite{DP-CMP}, with some important differences needed to handle the many-body nature of the problem. It is based on defining minimization problems with global constraints. Consider the following sets: for $a>0$, let 
 \begin{equation} \label{Ag}
\calA_g(a) = \Big\{\varrho\in\calE:\; \|n[\varrho]\|_{L^1}=a \Big\},
\end{equation}
and
\begin{equation} \label{Age}
\calA_{g,e}(a) = \Big\{\varrho\in\calE:\; \|e[\varrho]\|_{L^1}=a \Big\}.
\end{equation}
Starting from \fref{ineqS}, we will prove the following  crucial two inequalities: if
$$\|e[\varrho^\star]\|_{L^1}=b^\star \leq b=\|e[\varrho]\|_{L^1}, \qquad \|n[\varrho^\star]\|_{L^1}=a^\star \leq a=\|n[\varrho]\|_{L^1},$$
then
\be \label{inimp}
\inf_{\calA_{g}(a^\star)} F_\beta \leq \inf_{\calA_{g}(a)} F_\beta, \qquad \inf_{\calA_{g,e}(b^\star)} S \leq \inf_{\calA_{g,e}(b)} S,
\ee
where $F_\beta$ is the free energy at temperature $\beta^{-1}$,
%\be \label{freeE}
$$
F_\beta(\varrho)= \beta^{-1} S(\varrho) + \Tr (\HH^{1/2} \varrho \HH^{1/2} ), \qquad \varrho \in \calE.
$$
%\ee
We will prove that the minima of $F_\beta$ in $\calA_{g}(a)$ and of $S$ in $\calA_{g,e}(b)$ are achieved by using (global) Gibbs states of the form
$$
\varrho_{\alpha,\beta}=\frac{e^{- \HH_{\alpha,\beta}}}{\Tr (e^{- \HH_{\alpha,\beta}})},
$$
where for $(\alpha,\beta) \in \Rm_+ \times \Rm_+^*$, $\HH_{\alpha,\beta}= \beta \HH+\alpha \calN$.

Intuitively, \fref{inimp} is only possible if $b^\star=b$, $a^\star=a$, and if the inequalities are equalities. Call indeed $-S$ the physical entropy. If we accept the heuristics that the equilibrium physical entropy maximizes disorder, the equilibrium state with the largest energy should have the largest physical entropy, which contradicts the second inequality in \fref{inimp} if $b^\star<b$. In the same way, we expect the equilibrium state with the largest average number of particles to lose the largest amount of energy to thermal fluctuations, and therefore to have a lower equilibrium free energy than the equilibrium state with less particles. This contradicts the first inequality in \fref{inimp} if $a^\star<a$. An important part of the proof is to make these arguments rigorous.

Once we know that $\|e[\varrho^\star]\|_{L^1}=\|e_0\|_{L^1}$, and $\|n[\varrho^\star\|_{L^1}=\|n_0\|_{L^1}$, arguments for nonnegative operators of the type ``weak convergence plus convergence of the norms imply strong convergence'' lead to \fref{egalconst}.

Note that it is important to treat the density and energy constraints separately. If one were to set an optimization problem with global constraints both on the density and the energy, one would have to study the minimal possible energy of Gibbs states of the form $\varrho_{\alpha,\beta}$ for a fixed average number of particles. This is not direct as this requires to investigate the ground state of the Hamiltonian $\HH$ for the interaction potential $w$ for both the fermionic and bosonic cases. By separating the two constraints, we circumvent this issue and can then achieve arbitrary low energy states by simply decreasing the temperature and  the chemical potential (which is $-\alpha$ here).

In addition, it is necessary to treat the energy constraint first and to obtain that $\|e[\varrho^\star]\|_{L^1}=\|e_0\|_{L^1}$ before handling the density. This allows us to introduce the free energy in \fref{inimp}, which is bounded below and admits a minimizer on $\calA_{g}(a)$. For otherwise we would work with the problem
$$
\inf_{\calA_{g}(a)} S,
$$
which does not have a solution since $S$ is not bounded below on $\calA_{g}(a)$.\\

The rest of the article is dedicated to the proof of Theorem \ref{thm:Main}.

\section{Proof of the theorem} \label{proofT}

In Section \ref{secpropS}, we show that the entropy is bounded below and lower semi-continuous for states in the feasible set $\calA(n_0,u_0,e_0)$. Section \ref{secseq} consists in the core of the proof where we show that minimizing sequences converge in a strong sense. In Sections \ref{propproof1}, \ref{propproof2}, and \ref{propproof3}, we give the proofs of some important results that we were needed in Section \ref{secseq}. Finally, an appendix collects the proofs of some technical results.

\subsection{Properties of the entropy} \label{secpropS}
We will use the relative entropy between two states $\varrho$ and $\sigma$, which is defined by
$$
\calF(\varrho,\sigma)= \Tr \big(\varrho (\log \varrho-\log \sigma) \big) \in [0, \infty].
$$
It is set to the infinity when the kernel of $\sigma$ is not included in the kernel of $\varrho$. See e.g. \cite{wehrl} for more details about the relative entropy. We recall that $e^{-\beta h_c}$ is trace class for any $\beta>0$ by assumption \fref{GT}. According to \cite[Prop. 5.2.27]{brat2} for the bosonic case, and \cite[Prop. 5.2.22]{brat2} for the fermionic case, this implies that $e^{-\beta d\Gamma(h_c)}$ is trace class as well. These trace results are essentially the only parts in the proof where a distinction between fermions and bosons is made. Let then
  $$
\varrho_c=\frac{e^{-\beta d\Gamma(h_c)}}{ \Tr (e^{-\beta d\Gamma(h_c)})},
$$
 which is a state. We have the following result, which is a straightforward consequence of the nonnegativity of the relative entropy.

\begin{lemma} \label{Sbelow} Let $\varrho \in \calS$ with $\Tr\big( d\Gamma(h_c)^{1/2} \varrho\, d\Gamma(h_c)^{1/2} \big)$ finite. Then
  $$
  S(\varrho) \geq - \beta \Tr\big( d\Gamma(h_c)^{1/2} \varrho\, d\Gamma(h_c)^{1/2} \big) - \log \Tr (e^{-\beta d\Gamma(h_c)}).
  $$
\end{lemma}

\begin{proof}
  Let $\varrho$ satisfy the assumptions in the lemma. Then, 
$$
\calF(\varrho,\varrho_c)=S(\varrho)+\beta \Tr\big( (d\Gamma(h_c))^{1/2} \varrho (d\Gamma(h_c))^{1/2} \big)+ \log\Tr (e^{-d\Gamma(h_c)}) \geq 0,
$$
which proves the result. Note that there is a formal calculation that has to be justified, i.e. $\Tr(\varrho\, d\Gamma(h_c) )=\Tr\big( (d\Gamma(h_c))^{1/2} \varrho (d\Gamma(h_c))^{1/2} \big)$. When $\varrho\, d\Gamma(h_c)$  is not trace class, this is done by a regularization that we do not detail.
\end{proof}

\medskip

The next result  follows also directly from the properties of the relative entropy.

\begin{lemma} \label{contS} Let $\varrho \in \calS$, and consider a sequence of states such that $\varrho_m$ converges to $\varrho$ weak-$*$ in $\calJ_1$ as $m \to \infty$. Suppose moreover that there exists $C>0$ independent of $m$ such that
  $$
  \Tr\big( d\Gamma(h_c)^{1/2} \varrho\, d\Gamma(h_c)^{1/2} \big)+\Tr\big( d\Gamma(h_c)^{1/2} \varrho_m\, d\Gamma(h_c)^{1/2} \big) \leq C.$$
  Then
  $$
  S(\varrho) \leq \liminf_{m \to \infty} S(\varrho_m).
  $$
  \end{lemma}

\begin{proof} Write
$$
S(\varrho_m)=\calF(\varrho_m,\varrho_c)-\beta \Tr\big( (d\Gamma(h_c))^{1/2} \varrho_m (d\Gamma(h_c))^{1/2} \big)- \log\Tr (e^{-\beta d\Gamma(h_c)}),
$$
so that
$$
S(\varrho_m)\geq \calF(\varrho_m,\varrho_c)-C\beta - \log\Tr (e^{-\beta d\Gamma(h_c)}).
$$
According to \cite[Theorem 2]{LS}, the relative entropy is weakly lower semicontinuous, and therefore
\bee
\liminf_{m \to \infty} S(\varrho_m) &\geq& \calF(\varrho,\varrho_c)-C\beta - \log\Tr (e^{-\beta d\Gamma(h_c)})\\
&\geq & S(\varrho)+\beta \Tr\big( (d\Gamma(h_c))^{1/2} \varrho (d\Gamma(h_c))^{1/2} \big)-C\beta.
\eee
Sending $\beta$ to zero then yields the result.
\end{proof}
\begin{remark} Under the conditions of Lemma \ref{Sbelow}, we have in fact  that
$$
    S(\varrho) =\lim_{m \to \infty} S(\varrho_m).
    $$
Indeed, a direct adaptation of Lemma \ref{lem:comp} further shows that $\{\varrho_m\}_{m \in \NN}$ actually converges strongly to $\varrho$ in $\calJ_1$, and as a consequence that the eigenvalues $\{\rho_j^{(m)}\}_{j \in \NN}$ of $\varrho_m$ converge to those of $\varrho$, denoted $\{\rho_j\}_{j \in \NN}$, as $m \to \infty$. Fatou's lemma for sequences then yields
$$
\sum_{j\in \NN} - \rho_j \log \rho_j \leq  \liminf_{m \to \infty} \sum_{j\in \NN} - \rho^{(m)}_j \log \rho_j^{(m)},
$$
which corresponds to
$$
    S(\varrho) \geq \limsup_{m \to \infty} S(\varrho_m).
$$
\end{remark}
\subsection{Minimizing sequences} \label{secseq} The starting point is that $S$ is bounded below on $\calA(n_0,u_0,e_0)$. Indeed, estimate \fref{below2} shows that states in $\calA(n_0,u_0,e_0)$ satisfy the assumptions of Lemma \ref{Sbelow}, and as a consequence, using again \fref{below2} together with \fref{eqeneg},
$$
S(\varrho) \geq - \gamma^{-1}\Tr\big( \HH^{1/2} \varrho\, \HH^{1/2} \big) - \log \Tr (e^{-d\Gamma(h_c)})=-\gamma^{-1}\|e_0\|_{L^1}- \log \Tr (e^{-d\Gamma(h_c)}),$$
for all $\varrho \in \calA(n_0,u_0,e_0)$. There exists then a minimizing sequence $\{\varrho_m\}_{m\in \NN}$ in $\calA(n_0,u_0,e_0)$ such that
%\be \label{start}
$$
\lim_{m \to \infty}S(\varrho_m)=  \inf_{\calA(n_0,u_0,e_0)} S.
$$
%\ee

We have the following compactness result. 

 \begin{lemma} \label{lem:comp}
   Let $\{\varrho_m\}_{m \in \NN}$ be a sequence in $\calE$ with
   \be
   \label{boundlem}
   \Tr\left( \HH^{1/2} \varrho_m \,\HH^{1/2} \right) \leq C,
   \ee for some $C$ independent of $m$. Then, there exists $\varrho^\star \in \calE$ and a subsequence $\{\varrho_{m_j}\}_{j\in \NN}$ that converges to $\varrho^\star$ strongly in $\calJ_1$, and such that
\begin{align*}    
\mathcal{N}^{1/2}\varrho_{m_j}\mathcal{N}^{1/2} &\underset{j\to+\infty}\to \mathcal{N}^{1/2}\varrho^\star \mathcal{N}^{1/2},\quad \textrm{weak-$*$}\textrm{ in }\mathcal{J}_1,
\\ \HH^{1/2}\varrho_{m_j}\HH^{1/2}&\underset{j\to+\infty}\to \HH^{1/2}\varrho^\star \HH^{1/2},\quad \textrm{weak-$*$}\textrm{ in }\mathcal{J}_1,
\end{align*}
with
\bea \label{inflim}
\left\{\begin{array}{ll}
\ds \Tr\left( \calN^{1/2} \varrho^\star \,\calN^{1/2} \right)\leq \liminf_{j \to \infty} \Tr\left( \calN^{1/2} \varrho_{m_j} \,\calN^{1/2} \right),
\\ \ds \Tr\left( \HH^{1/2} \varrho^\star \,\HH^{1/2} \right)\leq \liminf_{j \to \infty} \Tr\left( \HH^{1/2} \varrho_{m_j} \,\HH^{1/2} \right).
\end{array}\right.
\eea
\end{lemma}

The proof of Lemma \ref{lem:comp} is classical and is given below for the reader's convenience.

\begin{proof} First of all, since $\varrho_m$ is a state, we have $\Tr (\varrho_m)=1$, which, together with \fref{boundlem}, and the fact that the space of trace class operators is the dual of the space of compact operators, implies that there exist $\varrho^\star \in \calJ_1$ and $\sigma \in \calJ_1$, and a subsequence such that $\varrho_{m_j}$ and $\HH^{1/2} \varrho_{m_j} \,\HH^{1/2}$ converge to $\varrho^\star$ and $\sigma$ in $\calJ_1$ weak-$*$ as $j \to \infty$, respectively. It is direct to identify $\sigma$: let $K$ compact in $\frakF_{b/f}$, and let $\BB=(\textrm{Id}+\HH)^{-1}$, which is bounded. Then:
\bee
\lim_{j \to \infty} \Tr (\HH^{1/2} \varrho_{m_j} \,\HH^{1/2} \BB K \BB)&=&\Tr (\sigma \BB K \BB)\\
&=& \lim _{j \to \infty} \Tr ( \varrho_{m_j} \,\HH^{1/2} \BB K \BB \HH^{1/2})\\
&=&\Tr (\varrho^\star \HH^{1/2}\BB K \BB \HH^{1/2}).
\eee
  In the last line, we used that $\HH^{1/2}\BB K \BB \HH^{1/2}$ is compact. This shows that $\sigma=\HH^{1/2} \varrho^\star \,\HH^{1/2}$. We proceed in the same way for the limit of $\calN^{1/2} \varrho_{m_j} \,\calN^{1/2}$.

  The limits in \fref{inflim} follow from the weak-$*$ convergence and the fact that $\Tr(\varrho)=\|\varrho\|_{\calJ_1}$ when $\varrho \geq 0$. 

Regarding the strong convergence to $\varrho^\star$ in $\calJ_1$, we claim first that $(\textrm{Id}+\HH)^{-1}$ is compact. Indeed, we have
$$
(\textrm{Id}+\HH)^{-1}=1 \oplus (\textrm{Id}_\frakh+h_0+V)^{-1}\oplus \bigoplus_{ n= 2}^{+\infty} (\textrm{Id}_{\frakF^{(n)}_{b/f}}+H_n)^{-1}.
$$
%The operator $h_c$ has a compact resolvant, see \cite[Theorem XIII.47]{RS-80-4}, and so do $H_n^c$. Then, for $n \geq 2$, each of the $H_n$ have compact resolvants as perturbations of $H_n^c$ and then so has $H_n$ as an application of Theorem 3.4, Chapter 6 \S 2 in \cite{kato}.
 
The operator $h_0+v_c$ has a compact resolvent according to \cite[Theorem XIII.47]{RS-80-4}, and so does $H_n^c$. Then, for $n \geq 2$, each of the $H_n$ have a compact resolvent as perturbations of $H_n^c$ as an application of \cite[Theorem 3.4, Chapter 6 \S 3]{kato}. To obtain that $(\textrm{Id}+\HH)^{-1}$ is compact, it just remains to show that $\|(\textrm{Id}_{\frakF^{(n)}_{b/f}}+H_n)^{-1}\|_{\calL(\frakF^{(n)}_{b/f})} \to 0$ as $n \to \infty$, see \cite[Theorem 4.1]{arai}. This is a consequence of \fref{Hbelow}, that yields
$$
\|(\textrm{Id}_{\frakF^{(n)}_{b/f}}+H_n)^{-1}\|_{\calL(\frakF^{(n)}_{b/f})} \leq (1+\gamma n)^{-1}.
$$
Second of all, it is not difficult to establish that the weak-$*$ convergence of $\varrho_{m_j}$ and $\HH^{1/2} \varrho_{m_j} \,\HH^{1/2}$ imply the weak-$*$ convergence of $(\textrm{Id}+\HH)^{1/2} \varrho_{m_j} \,(\textrm{Id}+\HH)^{1/2}$ to $(\textrm{Id}+\HH)^{1/2} \varrho^\star \,(\textrm{Id}+\HH)^{1/2}$. Then,

\bee
\lim_{j \to \infty} \Tr (\varrho_{m_j})&=&\Tr \Big( (\textrm{Id}+\HH)^{1/2} \varrho_{m_j} \,(\textrm{Id}+\HH)^{1/2} (\rm{Id}+\HH)^{-1} \Big)\\
&=& \Tr \Big( (\textrm{Id}+\HH)^{1/2} \varrho^\star \,(\textrm{Id}+\HH)^{1/2} (\rm{Id}+\HH)^{-1}\Big) \\
&=&\Tr (\varrho^\star).
\eee
Finally, according to \cite[Theorem 2.21, Addendum H]{Simon-trace}, weak convergence in sense of operators together with the convergence of the norm in $\calJ_1$ implies strong convergence in $\calJ_1$. Since weak-$*$ convergence in $\calJ_1$ implies weak convergence in the sense of operators, we obtain that $\varrho_{m_j}$ converges strongly to $\varrho^\star$ in $\calJ_1$. In particular, $\varrho^\star$ is a state. This ends the proof.
\end{proof}

  \medskip
  
We now continue the study of minimizing sequences. Since $\varrho_m$ is in $\calA(n_0,u_0,e_0)$ and therefore satisfies the constraints, we have, for all $m \in \NN$,
  \begin{align}
    &\Tr\left( \calN^{1/2} \varrho_m \,\calN^{1/2} \right)=\|n[\varrho_m]\|_{L^1}=\|n_0\|_{L^1} \label{equa1}\\
      &\Tr\left( \HH^{1/2} \varrho_m \,\HH^{1/2} \right)=\|e[\varrho_m]\|_{L^1}=\|e_0\|_{L^1}. \label{equa2}
  \end{align}
  According to Lemma \ref{lem:comp}, there exists then a subsequence (that we still denote abusively by $\{\varrho_m\}_{m\in \NN}$) that converges in the weak-$*$ topology of $\mathcal{J}_1$ to a state $\varrho^\star\in\calE$. Since the continuity result given in Lemma \ref{contS} shows that
  \be \label{ineqSS}
S(\varrho^\star) \leq \lim_{m \to \infty}S(\varrho_m)=  \inf_{\calA(n_0,u_0,e_0)} S,
\ee
we are left to prove that $\varrho^{\star}\in\mathcal{A}(n_0,u_0,e_{0})$, i.e. $\varrho^\star$ verifies the local constraints. For this, we have from Lemma \ref{lem:comp},
\begin{align} \label{infn}
\|n[\varrho^\star]\|_{L^1} = \mathrm{Tr}\left(\mathcal{N}^{1/2}\varrho^\star\mathcal{N}^{1/2}\right) &\leq \underset{m\to+\infty}{\liminf} \; \mathrm{Tr}\left(\mathcal{N}^{1/2}\varrho_m\mathcal{N}^{1/2}\right) = \|n_0\|_{L^1}
\end{align}
and
\begin{align} \label{infe}
\|e[\varrho^\star]\|_{L^1} = \mathrm{Tr}\left(\mathbb{H}^{1/2}\varrho^\star\mathbb{H}^{1/2}\right) &\leq \underset{m\to+\infty}{\liminf} \; \mathrm{Tr}\left(\mathbb{H}^{1/2}\varrho_m\mathbb{H}^{1/2}\right) =\|e_0\|_{L^1}.
\end{align}

As already mentioned in the introduction, a one-body version of Lemma \ref{lem:comp} yields directly that $n[\varrho^\star]=n_0$. It is not true in the many-body case. The issue is the following,  and is related to the identification of the 1-particle density matrices. From \fref{equa1} and \fref{equa2}, it is possible to show that the sequence of 1-particle density matrices $\varrho_m^{(1)}$ converges to some $\sigma^{(1)}$ strongly in $\calJ_1(\frakh)$. In particular, the local density of $\sigma^{(1)}$ is $n_0$. The difficulty is to identify $\sigma^{(1)}$ with the 1-particle density matrix of $\varrho^\star$, which is not possible at that stage. Indeed, we have, for all $\varphi \in L^\infty(\Rm^d)$,
\bee
\int_{\mathbb{R}^d} n[\varrho_m](x) \varphi(x) dx &=& \mathrm{Tr}_{\mathfrak{h}}\left(\varrho_m^{(1)} \varphi\right)\\
&=&\mathrm{Tr}\left(\calN^{1/2}\varrho_m \calN^{1/2}  \calN^{-1} d\Gamma(\varphi)\right).
\eee
The operator $\calN^{-1} d\Gamma(\varphi)$ is bounded in $\frakF_{b/f}$, but not compact, which does not allow us to pass to the limit above since $\calN^{1/2}\varrho_m \calN^{1/2}$ converges only weak-$*$ in $\calJ_1$. One could replace $\calN$ by $\HH$ above, but while the projections of $ \HH^{-1} d\Gamma(\varphi)$ on each $\frakF^{(n)}_{b/f}$ are compact, $ \HH^{-1} d\Gamma(\varphi)$ is not compact since these projections do not tend to zero in $\calL(\frakF^{(n)}_{b/f})$ as $n \to \infty$.\\

% \begin{equation*}
%  \lim_{m \to \infty} \int_{\mathbb{R}^d} n[\varrho_m](x) \varphi(x) dx = \mathrm{Tr}_{\mathfrak{h}}\left(\varrho_m^{(1)} \varphi\right)=\mathrm{Tr}\left(\varrho_m   \varphi\right),
% \end{equation*} 

We have then the following proposition, proved in Section \ref{propproof1}:

\begin{proposition} \label{propegal} Assume that $\|n[\varrho^\star]\|_{L^1}=\|n_0\|_{L^1}$ and that $\|e[\varrho^\star]\|_{L^1}=\|e_0\|_{L^1}$. Then $\varrho^\star \in \mathcal{A}(n_0,u_0,e_{0})$.
  \end{proposition}

  Based on this last result, the main difficulty is therefore to prove that $\|n[\varrho^\star]\|_{L^1}<\|n_0\|_{L^1}$ and $\|e[\varrho^\star]\|_{L^1}<\|e_0\|_{L^1}$ is not possible. We will use for this the next lemma, where $\calA_g(a)$ and $\calA_{g,e}(a)$ are defined in \fref{Ag} and \fref{Age}.

%$$
%\calG(n_0,u_0,e_0)=\inf_{\calA(n_0,u_0,e_0)} S.
%$$
\begin{lemma} \label{comS}(i) Suppose that the problem
  $$
  \inf_{\calA_g(a)} F_\beta
  $$
  admits a solution. Then,
  \be \label{eq1}
  \inf_{\calA_g(a)} F_\beta=\inf_{
    \scriptsize{\left\{\begin{array}{l}
                         (n,u,e) \in \calM\\
  n \in L^1_+ \textrm{with} \; \|n\|_{L^1}=a  \\
  u \in (L^1)^d, e \in L^1_+
\end{array} \right.}
  } \inf_{\calA(n,u,e)} F_\beta.
  \ee

  (ii) Suppose that the problem
  $$
  \inf_{\calA_{g,e}(a)} S
  $$
  admits a solution. Then,
  $$
  \inf_{\calA_{g,e}(a)} S=\inf_{
    \scriptsize{\left\{\begin{array}{l}
                         (n,u,e) \in \calM\\
  e \in L^1_+ \textrm{with} \; \|e\|_{L^1}=a  \\
  u \in (L^1)^d, n \in L^1_+
\end{array} \right.}
  } \inf_{\calA(n,u,e)} S.
  $$
\end{lemma}
\begin{proof} We start with (i). Let $(n,u,e) \in \calM$ and $\varrho \in \calA(n,u,e)$ with $\|n\|_{L^1}=a$. Then $\varrho \in \calA_g(a)$, and as a consequence
  $$
  \inf_{\calA_g(a)} F_\beta \leq  \inf_{\calA(n,u,e)} F_\beta.
  $$
Taking the infimum as in \fref{eq1} then yields a first inequality in \fref{eq1}. For the reverse inequality, denote by $\sigma$ a minimizer of $F_\beta$ in $\calA_g(a)$ and let
$$
G(n,u,e)=\inf_{\calA(n,u,e)} F_\beta.
$$
By construction $\|n[\sigma]\|_{L^1}=a$, $u[\sigma] \in (L^1(\Rm^d))^d$ and $e[\sigma] \in L^1_+(\Rm^d)$, and clearly $(n[\sigma],u[\sigma],e[\sigma])$ is admissible. Hence,
$$
\inf_{
  \scriptsize{\left\{\begin{array}{l}
                       (n,u,e) \in \calM\\
  n \in L^1_+ \textrm{with} \; \|n\|_{L^1}=a  \\
  u \in (L^1)^d, e \in L^1_+
\end{array} \right.}
}G(n,u,e)
\leq G(n[\sigma],u[\sigma],e[\sigma]).
$$
It remains to prove that
$$
G(n[\sigma],u[\sigma],e[\sigma])=\min_{\calA_g(a)} F_\beta,
$$
which is straightforward since
$$
  \min_{\calA_g(a)} F_\beta \leq  G(n[\sigma],u[\sigma],e[\sigma])=\inf_{\calA(n[\sigma],u[\sigma],e[\sigma])} F_\beta \leq F_\beta(\sigma)=\min_{\calA_g(a)} F_\beta.
  $$
  This proves (i). Item (ii) follows in the same manner by replacing $n$ by $e$ and $F_\beta$ by $S$. This ends the proof.
\end{proof}

\bigskip

Let now $\|e[\varrho^\star]\|_{L^1}=b^\star$ and $\|e_0\|_{L^1}=b$. Since $\varrho^\star \in \calE$, we have
 $$
\inf_{\calA_{g,e}(b^\star)} S \leq S(\varrho^\star).
$$
Together with \fref{ineqSS}, this gives
$$
\inf_{\calA_{g,e}(b^\star)} S \leq \inf_{\calA(n_0,u_0,e_0)} S.
$$
Assuming for the moment that $S$ admits a minimizer on $\calA_{g,e}(b)$, item (ii) of Lemma \ref{comS} yields
$$
\inf_{\calA_{g,e}(b^\star)} S \leq 
\inf_{
  \scriptsize{\left\{\begin{array}{l}
                       (n,u,e) \in \calM\\
  e \in L^1_+ \textrm{with} \; \|e\|_{L^1}=\|e_0\|_{L^1}  \\
  u \in (L^1)^d, n \in L^1_+
\end{array} \right.}
  } \inf_{\calA(n,u,e)} S=\inf_{\calA_{g,e}(b)} S,
$$
which results in
\be \label{cont1}
\inf_{\calA_{g,e}(b^\star)} S \leq \inf_{\calA_{g,e}(b)} S.
\ee

The next result shows that there is a contradiction above if $b^\star<b$.

\begin{proposition} \label{gminS}
  Let $a>0$. Then, the minimization problem

  $$
  \inf_{\calA_{g,e}(a)} S
  $$
  admits a unique solution. Let moreover $f(a)=\inf_{\calA_{g,e}(a)} S$. Then, $f$ is a strictly decreasing continuous function on $\Rm_+$.
\end{proposition}

The proof of Proposition \ref{gminS} is given in Section \ref{propproof2}. Suppose that $\|e[\varrho^\star]\|_{L^1}=b^\star<b=\|e_0\|_{L^1}$. Based on the previous proposition, we have $f(b)<f(b^\star)$, which contradicts \fref{cont1}, and  we obtain therefore the equality
\be \label{eqE}\|e[\varrho^\star]\|_{L^1}=\|e_0\|_{L^1}.
\ee

It remains to prove that
\be \label{eqn}\|n[\varrho^\star]\|_{L^1}=\|n_0\|_{L^1}.
\ee
For this, let $\beta>0$. From \fref{ineqSS} and \fref{eqE}, we find
\be \label{ineqN2}
F_\beta(\varrho^\star)=\beta^{-1} S(\varrho^\star)+ \|e[\varrho^\star]\|_{L^1} \leq \beta^{-1} \inf_{\calA(n_0,u_0,e_0)} S+  \|e_0\|_{L^1}
=\inf_{\calA(n_0,u_0,e_0)} F_\beta .\ee

The next proposition is similar to Propositon \ref{gminS} and is proved in Section \ref{propproof3}.

\begin{proposition} \label{gminF}
  Let $a_0>0$. Then, there exists $\beta_0(a_0)>0$, such that the minimization problem
$$
  \inf_{\calA_{g}(a)} F_\beta
  $$
  admits a unique solution for any $\beta \leq \beta_0(a)$ and any $a \leq a_0$. Let moreover $g(a)=\inf_{\calA_{g}(a)} F_\beta$. Then, $g$ is a strictly decreasing continuous function on $(0,a_0]$.
\end{proposition}

We apply Proposition \ref{gminF} as follows: let $\|n[\varrho^\star]\|_{L^1}=a^\star$ and $\|n_0\|_{L^1}=a$. We have then from \fref{ineqN2},
$$
\inf_{\calA_g(a^\star)} F_\beta  \leq F_\beta(\varrho^\star)\leq \inf_{\calA(n_0,u_0,e_0)} F_\beta .$$
In Proposition \ref{gminF}, choose $a_0=\|n_0\|_{L^1}$. Since $\|n[\varrho^\star]\|_{L^1} \leq \|n_0\|_{L^1}$, then both minimization problems on $\calA_{g}(a)$ and $\calA_{g}(a^\star)$ admit a unique solution for $\beta \leq \beta_0(a_0)$.
Then, according to item (i) of Lemma \ref{comS},
$$
\inf_{\calA_{g}(a^\star)} F_\beta \leq 
\inf_{
  \scriptsize{\left\{\begin{array}{l}
                       (n,u,e) \in \calM\\
  n \in L^1_+ \textrm{with} \; \|n\|_{L^1}=\|n_0\|_{L^1}  \\
  u \in (L^1)^d, e \in L^1_+
\end{array} \right.}
  } \inf_{\calA(n,u,e)} F_\beta=\inf_{\calA_{g}(a)} F_\beta,
$$
which results in
$$
\inf_{\calA_{g}(a^\star)} F_\beta \leq \inf_{\calA_{g}(a)} F_\beta .
$$
But based on the previous proposition, we have $g(a)<g(a^\star)$, which is a contradiction if $a^\star<a$. We therefore obtain \fref{eqn}, which ends the proof of Theorem \ref{thm:Main} as an application of Proposition \ref{propegal}.

  \subsection{Proof of Proposition \ref{propegal}} \label{propproof1} When $\|n[\varrho^\star]\|_{L^1}=\|n_0\|_{L^1}$ and $\|e[\varrho^\star]\|_{L^1}=\|e_0\|_{L^1}$, we obtain from \fref{infn} and $\fref{infe}$ that 
\begin{align*}
\mathrm{Tr}\left(\mathcal{N}^{1/2}\varrho^\star\mathcal{N}^{1/2}\right) =\underset{m\to+\infty}{\lim} \; \mathrm{Tr}\left(\mathcal{N}^{1/2}\varrho_m\mathcal{N}^{1/2}\right) 
\end{align*}
and
\begin{align*}
\mathrm{Tr}\left(\mathbb{H}^{1/2}\varrho^\star\mathbb{H}^{1/2}\right) =\underset{m\to+\infty}{\lim} \; \mathrm{Tr}\left(\mathbb{H}^{1/2}\varrho_m\mathbb{H}^{1/2}\right).
\end{align*}
According to \cite[Theorem 2.21, Addendum H]{Simon-trace}, this implies, together with the corresponding weak-$*$ convergences, that $\mathcal{N}^{1/2}\varrho_m\mathcal{N}^{1/2}$ and $\mathbb{H}^{1/2}\varrho_m\mathbb{H}^{1/2}$ converge strongly in $\calJ_1$ to $\mathcal{N}^{1/2}\varrho^\star \mathcal{N}^{1/2}$ and $\mathbb{H}^{1/2}\varrho^\star \mathbb{H}^{1/2}$.

Let now $\varrho_m^{(1)}$ and $\varrho^\star_{(1)}$ be the one-particle density matrices of  $\varrho_m$ and $\varrho^\star$. We will show that $\varrho_m^{(1)}$ converges to $\varrho^\star_{(1)}$ strongly in $\calJ_1(\frakh)$. Indeed, with the definition \fref{defB},
\bee
\| \varrho_m^{(1)}- \varrho^\star_{(1)}\|_{J_1(\frakh)}&=&\sup_{\| A \|_{\calL(\frakh)} \leq 1} \Tr_\frakh \big((\varrho_m^{(1)}- \varrho^\star_{(1)}) A \big )\\
&=&\sup_{\| A \|_{\calL(\frakh)} \leq 1} \Tr \big((\varrho_m- \varrho^\star) \AA \big )\\
&=&\sup_{\| A \|_{\calL(\frakh)} \leq 1} \Tr \big(\calN^{1/2}(\varrho_m- \varrho^\star)\calN^{1/2} \BB \big )\\
&\leq& \|\calN^{1/2}\varrho_m\calN^{1/2}- \calN^{1/2}\varrho^\star\calN^{1/2}\|_{\calJ_1},
\eee
which yields the result.

We prove similarly that $h_0^{1/2}\varrho_m^{(1)}h_0^{1/2}$ converges to $h_0^{1/2}\varrho^\star_{(1)}h_0^{1/2}$ strongly in $\calJ_1(\frakh)$. An easy consequence of this and of the strong convergence of $\varrho_m^{(1)}$ in $\calJ_1(\frakh)$, is that $(h_0+\textrm{Id}_\frakh)^{1/2}\varrho_m^{(1)}(h_0+\textrm{Id}_\frakh)^{1/2}$ converges to $(h_0+\textrm{Id}_\frakh)^{1/2}\varrho^\star_{(1)}(h_0+\textrm{Id}_\frakh)^{1/2}$ strongly in $\calJ_1(\frakh)$.

We are now in position to identify $k[\varrho^\star]$ and $u[\varrho^\star]$. For $\varphi \in L^\infty(\Rm^d)$, let
$$
A=(h_0+\textrm{Id}_\frakh)^{-1/2}\nabla \cdot (\varphi \nabla)(h_0+\textrm{Id}_\frakh)^{-1/2},
$$
which is bounded in $\calL(\frakh)$. We have
from the definition of $k[\varrho^\star]$,
\begin{align*}
  \|k[\varrho_m] -&k[\varrho^\star] \|_{L^1}\\
  &=\sup_{\| \varphi \|_{L^\infty} \leq 1} \Tr_\frakh \big(\nabla \cdot (\varphi \nabla)(\varrho_m^{(1)}-\varrho^\star_{(1)}) \big )\\
&=\sup_{\| \varphi \|_{L^\infty} \leq 1} \Tr_\frakh \big(A (h_0+\textrm{Id}_\frakh)^{1/2}(\varrho_m^{(1)}-\varrho^\star_{(1)})(h_0+\textrm{Id}_\frakh)^{1/2} \big )\\
&\leq C \|(h_0+\textrm{Id})^{1/2}\varrho^{(1)}_m(h_0+\textrm{Id}_\frakh)^{1/2}- (h_0+\textrm{Id})^{1/2}\varrho^\star_{(1)}(h_0+\textrm{Id}_\frakh)^{1/2}\|_{\calJ_1(\frakh)}.
\end{align*}
This implies that $k[\varrho_m]$ converges to $k[\varrho^\star]$ strongly in $L^1(\Rm^d)$, and since $k[\varrho_m]=k_0$ by construction, we have $k[\varrho^\star]=k_0$. The proof that $u[\varrho^\star]=u_0$ is similar.

We have therefore obtained that $\varrho^\star$ satisfies the local constraints, and therefore that it belongs to the feasible set. This ends the proof.

  \subsection{Proof of Proposition \ref{gminS}} \label{propproof2}
In preparation of the proof, let, for $(\alpha,\beta) \in \Rm_+\times \Rm_+^*$, $\HH_{\alpha,\beta}= \beta \HH+\alpha \calN$. The operator $\HH_{\alpha,\beta}$ is self-adjoint on $D(\HH)$. According to \fref{Hbelow}, we have
\be \label{belowHa}
\gamma \beta d\Gamma(h_c)\leq \gamma \beta (d\Gamma(h_c)+\calN) \leq \HH_{\alpha,\beta}
\ee
in the sense of operators. Hence, $e^{- \HH_{\alpha,\beta}}$ is trace class for all $(\alpha,\beta) \in \Rm_+\times \Rm_+^*$ since $e^{-\gamma \beta d\Gamma(h_c)}$ is trace class as explained in Section \ref{secpropS}, and the partition function
$$
Z_{\alpha,\beta}=\Tr (e^{- \HH_{\alpha,\beta}})
$$
is well-defined for all $\alpha \geq 0$. Let then the Gibbs state
$$
\varrho_{\alpha,\beta}=\frac{e^{- \HH_{\alpha,\beta}}}{Z_{\alpha,\beta}}.
$$
Note that $\varrho_{\alpha,\beta}$ commutes with $\calN$. The associated average particle number and average energy are defined by, respectively,
$$
N(\alpha,\beta)=\Tr (\calN \varrho_{\alpha,\beta}), \qquad E(\alpha,\beta)=\Tr ( \HH  \varrho_{\alpha,\beta}).
$$
They are both well-defined for all $(\alpha,\beta) \in \Rm_+\times \Rm_+^*$ since in particular, according to \fref{Hbelow},
\bea \nonumber
Z_{\alpha,\beta} \Tr (\calN \varrho_{\alpha,\beta}) &=& 1+\sum_{n=1}^\infty n \Tr_{\frakF^{(n)}_{b/f}} (e^{- \beta H_n-\alpha n})
\leq 1+\sum_{n=1}^\infty n \Tr_{\frakF^{(n)}_{b/f}} (e^{- \gamma \beta H^c_n-\gamma n})\\
&\leq &1+C\sum_{n=1}^\infty \Tr_{\frakF^{(n)}_{b/f}} (e^{- \gamma \beta H^c_n})=C\Tr (e^{-\gamma \beta d\Gamma(h_c) }). \label{alz}
\eea

We will use the following lemma, helpful as well in the proof of Proposition \ref{gminF}.

\begin{lemma} \label{lemZ}
(i) $Z_{\alpha,\beta} \in C^\infty(\Rm_+^* \times \Rm^*_+)$ and
\be \label{expN}
N(\alpha,\beta)=- \partial_\alpha \log Z_{\alpha,\beta}, \qquad E(\alpha,\beta)=- \partial_\beta \log Z_{\alpha,\beta}.
\ee
  (ii) For all $\alpha \geq 0$, the function $ \Rm_+^* \ni \beta \mapsto E(\alpha,\beta)$ is continuously strictly decreasing with
  $$
  \lim_{\beta\to + \infty} E(\alpha,\beta)=0, \qquad \lim_{\beta\to 0} E(\alpha,\beta)=+\infty, \qquad \forall \alpha \geq 0.
  $$
  (iii) For all $\beta > 0$, the function $\Rm^*_+  \ni  \alpha \mapsto N(\alpha,\beta)$ is continuously strictly decreasing  with
  $$
  \lim_{\alpha\to + \infty} N(\alpha,\beta)=0, \qquad \forall \beta > 0, \qquad \lim_{\beta\to  0} N(0,\beta)=+\infty.
  $$
\end{lemma}
  %$$
  %\end{lemma}

\begin{proof} The proof is simplest by using the spectral decomposition of $\HH_{\alpha,\beta}$ instead of each of the $H_n$, $n \geq 2$. We have already observed in the proof of Lemma \ref{lem:comp} that $\HH$ has a compact resolvent, and therefore so does $\HH_{\alpha,\beta}$. Let $\{\psi_\ell\}_{\ell \in \NN}$ be a basis of $\frakF_{b/f}$ of eigenvectors of $\HH$. Then \fref{belowHa} for $\alpha=0$ shows that $\psi_\ell \in D(\calN)$. Then, since $\calN$ and $\HH$ commute, we can pick the basis $\{\psi_\ell\}_{\ell \in \NN}$ such that
  $$
  \calN \psi_\ell=n_\ell \psi_\ell, \qquad \ell \in \NN,
  $$
for some $n_\ell \geq 0$.  The eigenvalues of $\HH_{\alpha,\beta}$ are as a consequence of the form $\beta \lambda_\ell+ \alpha n_\ell$, for $\{\lambda_\ell\}_{\ell \in \NN}$ the eigenvalues of $\HH$. The $\lambda_\ell$'s are arranged into a nondecreasing sequence that tends to the infinity. The zero eigenvalue is simple and associated with the vacuum eigenvector $1 \oplus_{n=1}^\infty 0$, and we set $\lambda_0=n_0=0$. We have as well   $n_\ell \geq 1$ and $\lambda_\ell \geq \gamma$ for $\ell \geq 1$ according to \fref{Hbelow}. Of course, the  $\lambda_\ell$ and $n_\ell$ are different for fermions and bosons.
  
  % % Moreover, \fref{Hbelow} implies that
    % $$
    % (\beta \gamma +\alpha) \calN + \beta d\Gamma(h_c) \leq \HH_{\alpha_\beta},
    % $$
    % in the sense of operators. For $\{\zeta_\ell\}_{\ell \in \Nm}$ the eigenvalues of the second quantization of $h_c$, bosonic or fermionic, we have therefore $\beta \lambda_\ell +\alpha n_\ell \geq \beta \zeta_\ell + n_\ell(\beta \gamma+\alpha)$ for $n \geq 1 $, with $\sum_{\ell \in \Nm} e^{- \beta \zeta_\ell}$ finite for all $\beta>0$ since $e^{-\beta h_c}$ is trace class. We have as well $\lambda_0=n_0=0$. The zero eigenvalue is simple and associated with the vacuum eigenvector $1 \oplus_{n=1}^\infty 0$.
     The partition function $Z_{\alpha,\beta}$ then reads
    $$
Z_{\alpha,\beta}=\sum_{\ell \in \NN} e^{-\beta \lambda_\ell -\alpha n_\ell}>1.
$$
It is then direct to show, using dominated convergence for series, that $Z_{\alpha,\beta}$ is continuously infinitely differentiable for $(\alpha,\beta)\in \Rm^*_+ \times \Rm_+^*$. The expressions in \fref{expN} follow then easily, and as a consequence $N$ and $E$ are also continuously infinitely differentiable for $(\alpha,\beta)\in \Rm^*_+ \times \Rm_+^*$. In particular,

    $$\ds
    \partial_\alpha N(\alpha,\beta)=-\frac{\left(\sum_{\ell \in \NN} n^2_\ell e^{-\beta \lambda_\ell -\alpha n_\ell}\right) \left(\sum_{\ell \in \NN} e^{-\beta \lambda_\ell -\alpha n_\ell}\right)-\left(\sum_{\ell \in \NN} n_\ell e^{-\beta \lambda_\ell -\alpha n_\ell}\right)^2}{Z_{\alpha,\beta}^2}.
    $$
    The Cauchy-Schwarz inequality then shows that $\partial_\alpha N(\alpha,\beta)<0$ since the equality case is not possible as there exist some indices $\ell$ and $\ell'$ for which $n_\ell \neq n_{\ell'}$ (this is easily seen by remarking for instance that if $n_\ell= \bar n <\infty$ for all $\ell \geq 1$, then $\|\calN \psi\|=\bar n \|\psi\|$ for all $\psi \in D(\calN)$, which is absurd since $\calN$ is unbounded). A similar calculation shows that $\partial_\beta E(\alpha,\beta)<0$ for $(\alpha,\beta)\in \Rm^*_+ \times \Rm_+^*$.

    The limits as $\alpha \to +\infty$ for $N$ and as $\beta \to +\infty$ for $E$ follow easily by dominated convergence. For the limit as $\beta \to 0$ of $E$, set $\ell_M$ such that $\lambda_{\ell_M} \geq M$ for some $M>0$ fixed. Since $\sum_{\ell \geq \ell_M} e^{-\beta \lambda_\ell -\alpha n_\ell}$ converges to $+\infty$ as $\beta \to 0$ (for otherwise the operator $e^{-\alpha \calN}$ would be trace class), and since $\sum_{\ell < \ell_M} e^{-\beta \lambda_\ell -\alpha n_\ell}\leq \sum_{\ell < \ell_M}=\ell_M$, there exists $\beta_0(M)$ such that
    $$
     \sum_{\ell < \ell_M} e^{-\beta \lambda_\ell -\alpha n_\ell} \leq \sum_{\ell \geq \ell_M} e^{-\beta \lambda_\ell -\alpha n_\ell}, \qquad \forall \beta\leq \beta_0(M).
    $$
    Hence,

    $$
  \frac{M}{2} \leq \frac{1}{2} \frac{M\sum_{\ell \geq \ell_M} e^{-\beta \lambda_\ell -\alpha n_\ell} }{\sum_{\ell \geq \ell_M} e^{-\beta \lambda_\ell -\alpha n_\ell}} \leq E(\alpha,\beta),
    $$
    which proves the second limit in (ii).

    It remains to treat the second limit in (iii). Since the operator $\calN$ is not bounded on $\frakF_{b/f}$, we can rearrange the sequence $\{n_\ell\}_{\ell \in \NN}$ into a nondecreasing sequence $\{ n_{\ell_j}\}_{j \in \NN}$ such that $n_{\ell_j} \to \infty$ as $j \to \infty$. Proceeding as in the proof of the second limit in (ii), we then find that for all $M>0$, there exist $j_M$ and $\beta_0(M)$ such that
$$
\frac{M}{2} \leq \frac{1}{2} \frac{M \sum_{ j \geq j_M} e^{-\beta\lambda_{\ell_j}}}{\sum_{ j \geq j_M} e^{-\beta \lambda_{\ell_j}}} \leq \frac{\sum_{ j \in \NN} n_{\ell_j}e^{-\beta\lambda_{\ell_j}}}{\sum_{ j \in \NN} e^{-\beta\lambda_{\ell_j}}}=N(0,\beta), \qquad \forall \beta \leq \beta_0(M).
$$
This proves the second limit in (iii) and ends the proof of the lemma.
\end{proof}
\bigskip

We proceed now to the proof of Proposition \ref{gminS}. First of all, according to Lemma \ref{lemZ} (ii), there exists, for all $a>0$, a $\beta_0(a) \in (0,\infty)$ such that
  \be \label{egE}
  E(0,\beta_0(a))=a.
  \ee
  Since $E(0,\beta)$ is continuously differentiable and strictly monotone, the global version of the implicit function theorem implies that $a \mapsto \beta_0(a)$ is continuously differentiable. 
  Let now $\varrho \in \calE$ and consider the free energy
  $$
  F_{\beta_0(a)}(\varrho)=\beta^{-1}_0(a) S(\varrho)+ \Tr (\HH^{1/2} \varrho \HH^{1/2} ).
  $$
For $\calF$ the relative entropy introduced in Section \ref{secpropS}, we find
  $$
  \beta_0(a) F_{\beta_0(a)}(\varrho)=\calF(\varrho,\varrho_{0,\beta_0(a)})-\log Z_{0,\beta_0(a)},
  $$
  and as a consequence, for  any $\varrho \in \calA_{g,e}(a)$,
  $$
  S(\varrho)=\calF(\varrho,\varrho_{0,\beta_0(a)})-\log Z_{0,\beta_0(a)}-a \beta_0(a).
  $$
  Since $\calF(\varrho,\varrho_{0,\beta_0(a)})=0$ if and only if $\varrho=\varrho_{0,\beta_0(a)}$, we obtain that $\varrho_{0,\beta_0(a)}$ is the unique minimizer of $S$ in $\calA_{g,e}(a)$. This proves the existence part of Proposition \ref{gminS}.

  Regarding the monotonicity of the entropy, we have
  $$
  f(a)=S(\varrho_{0,\beta_0(a)})=-\log Z_{0,\beta_0(a)}-a \beta_0(a),$$
  which is continuously differentiable w.r.t. $a$ since $Z_{0,\beta}$ and $\beta_0(a)$ are both $C^1$ w.r.t $\beta$ and $a$, respectively. Then, thanks to Lemma \ref{lemZ} (i) and \fref{egE},
  $$
  f'(a)=\beta_0'(a) E(0,\beta_0(a))-\beta'_0(a) a -\beta_0(a)=-\beta_0(a)<0.
  $$
  This ends the proof.

\subsection{Proof of Proposition \ref{gminF}} \label{propproof3}

  The proof is very similar to that of Proposition \ref{gminS}.
  First of all, according to Lemma \ref{lemZ} (iii), we have
  $$
  N(\alpha,\beta) \leq N(0,\beta), \qquad \forall (\alpha,\beta) \in \Rm_+ \times \Rm_+^*.
  $$
Note that $N(0,\beta)$ is well-defined for $\beta>0$ as proved in \fref{alz}, and that $N(\alpha,\beta)$ is continuous at $\alpha=0$ for all $\beta>0$ as a direct application of monotone convergence. Pick $a_0>0$. We will set $\beta$ sufficiently small so that all $a>0$ less than $a_0$ are in the range of $N(\alpha,\beta)$ for all $\alpha$ sufficiently large.   

 Since $N(0,\beta) \to +\infty$ as $\beta \to 0$ according to Lemma \ref{lemZ} (iii), there exists $\beta_0(a_0)$ such that
  $$
  N(0,\beta) \geq a_0, \qquad \forall \beta \leq \beta_0(a_0).
  $$
We fix from now on a $\beta>0$ such that $\beta \leq \beta_0(a_0)$. Then, since $\alpha \mapsto N(\alpha,\beta)$ is strictly decreasing, there exists $\alpha_1>0$ such that
  $$
  N(\alpha,\beta) \leq a_0, \qquad \forall \alpha \geq \alpha_1,
  $$
 and for all $a \in (0,a_0]$, there exists $\alpha_0(a) \in [\alpha_1,\infty)$ such that
  \be \label{egN}
  N(\alpha_0(a),\beta)=a. \qquad 
  \ee
  Since $N(\alpha,\beta)$ is continuously differentiable and strictly monotone in $\alpha$, the global version of the implicit function theorem implies that $\alpha \mapsto \alpha_0(a)$ is continuously differentiable. As in the proof of Proposition \ref{gminS}, we find that, for  any $\varrho \in \calA_{g}(a)$,
  \bee
  \beta F_\beta(\varrho)&=&\calF(\varrho,\varrho_{\alpha_0(a),\beta})-\log Z_{\alpha_0(a),\beta}-\alpha_0(a) \Tr (\calN^{1/2} \varrho \calN^{1/2})\\
&=& \calF(\varrho,\varrho_{\alpha_0(a),\beta})-\log Z_{\alpha_0(a),\beta}-\alpha_0(a) a,
  \eee
  which shows that $\varrho_{\alpha_0(a),\beta}$ is the unique minimizer of $F_\beta$ in $\calA_{g}(a)$.

  Regarding the monotonicity of the entropy, we have 
  $$
  \beta g(a)=\beta F_\beta(\varrho_{\alpha_0(a),\beta})=-\log Z_{\alpha_0(a),\beta}-a \alpha_0(a),$$
  which is continuously differentiable w.r.t. $a$ since $Z_{\alpha,\beta}$ and $\alpha_0(a)$ are both $C^1$ w.r.t $\alpha$ and $a$, respectively. Then, thanks to Lemma \ref{lemZ} (i) and \fref{egN},
   $$
  \beta g'(a)=\alpha_0'(a) N(\alpha_0(a),\beta)-a \alpha'_0(a)-\alpha_0(a)=-\alpha_0(a)<0.
  $$
  This ends the proof.
  
% We now recall the following result from \cite[Theorem 2.21, Addendum H]{simon2010trace}.
% \begin{theorem}\label{thm:Simon}
% Let $\{\varrho_m\}_{m\geq0}$ such that $\varrho_m\to\varrho$ weakly in the sense of operators. Then, if $\|\varrho_m\|_{\calJ_1} \to \|\varrho\|_{\calJ_1}$, we have that $\|\varrho_m\ -\varrho\|_{\calJ_1}\to 0$.
% \end{theorem}

\section{Appendix}
\subsection{Justification of Definition \ref{onebody}}

The fact that $\varrho^{(1)}$ is well-defined is established as follows. Set 
%\be \label{defB2}
$$
\BB=  0 \oplus \bigoplus_{n = 1}^{+\infty} \mathbb{B}_{(n)}:=0 \oplus \bigoplus_{n = 1}^{+\infty} n^{-1}\mathbb{A}_{(n)},
$$
where $\AA$ is the second quantization of $A$ and $\AA_{(n)}$ its component on the $n$-th sector.
%\ee
Then, we have the estimate
\be \label{boundB}
\| \BB \|_{\calL(\mathfrak{F}_{b/f})} \leq \|A\|_{\calL(\frakh)}.
\ee
Indeed, for $\psi =\{\psi^{(n)}\}_{n\in \NN}\in \mathfrak{F}_{b/f}$ and $n \geq 1$, we have
$$
\BB_{(n)} \psi^{(n)}=(\BB \psi)_{(n)}=\frac{1}{n} \left(\sum_{i=1}^n (A)_{i} \psi^{(n)} \right),
$$
for $(A)_{i}$ the operator $A$ acting on the variable $x_i$. This yields directly
$$
\| \BB_{(n)} \psi^{(n)}\|_{n} \leq \|A\|_{\calL(\frakh)} \| \psi^{(n)}\|_{n}.
$$
Since $\| \BB \|_{\calL(\mathfrak{F}_{b/f})}=\sup_{n \in \NN} \|\BB_{(n)}\|_{n}$, this gives \fref{boundB}. 

For $\varrho \in \calS_0$, consider now the linear  map
$$
F: A \in \calL(\frakh) \mapsto F(A)=\Tr( \AA \varrho).
$$
It is well-defined since, using the above notation for $\BB$,
$$
\left|\Tr( \AA \varrho) \right|=\left|\mathrm{Tr} \left (\BB \calN^{1/2} \varrho\, \calN^{1/2}\right)\right| \leq C \|A\|_{\calL(\frakh)},
$$
where we used the fact that 
$$
\mathrm{Tr} \left (\calN^{1/2} \varrho\, \calN^{1/2}\right)=C<\infty,
$$
since $\varrho \in \calS_0$. When $A$ is compact, $F$ is therefore a linear continuous map on the space of compact operators on $\frakh$. By duality, we can then conclude that there exists a unique $\varrho^{(1)} \in \calJ_1(\frakh)$ such that
$$
F(A)=\mathrm{Tr}_{\frakh}\left(A \varrho^{(1)}\right),
$$
for all compact operators $A$ on $\calL(\frakh)$. The case $A$ bounded follows finally by approximation.

The fact that $\varrho^{(1)}$ is nonnegative is established as follows. Let $\varphi \in \frakh$ with $\|\varphi\|_{\frakh}=1$, and consider the rank one projector $P=\ket{\varphi}\bra{\varphi}$. Then
$$
(\varphi,\varrho^{(1)} \varphi)_\frakh=\mathrm{Tr}_{\frakh}\left(P \varrho^{(1)}\right) =\mathrm{Tr} \left( d\Gamma(P) \varrho\right).
$$
Denoting by $\{\rho_p\}_{p \in \NN}$ and $\{\psi_p\}_{p \in \NN}$ the (nonnegative) eigenvalues and eigenfunctions of $\varrho \in \calE_0$, the last term is equal to 
  $$
  \sum_{p \in \NN} \sum_{n \in \NN^*} \sum_{j=1}^n \rho_p \left(\psi_p^{(n)}, P_{(j)}\psi_p^{(n)}\right)_{n}= \sum_{p \in \NN} \sum_{n \in \NN^*} \sum_{j=1}^n \rho_p \left(P_{(j)} \psi_p^{(n)}, P_{(j)}\psi_p^{(n)}\right)_{n} \geq 0.
  $$
  Above $P_{(j)}$ is the operator $P$ acting on $x_j$, and we used the fact that $P^2_{(j)}=P_{(j)}$. This yields the positivity and ends the justification of Definition \ref{onebody}.

\subsection{Proof of Lemma \ref{enrho}} 
  Denote by $\{\rho_p\}_{p \in \NN}$ and $\{\psi_p\}_{p \in \NN}$ the eigenvalues and eigenfunctions of $\varrho \in \calE_0$. A direct calculation shows first that
  $$
  \Tr\left( \HH_0^{1/2} \varrho \,\HH_0^{1/2} \right)=\sum_{p \in \NN} \sum_{n \in \NN^*} \rho_p \left\| d\Gamma(h_0)^{1/2}_{(n)} \psi_p^{(n)}\right\|^2_{n},
  $$
  where $\psi_p^{(n)}$ is the component of $\psi_p$ on the $n$-th sector and
  $$
  d\Gamma(h_0)^{1/2}_{(n)} = \left(\sum_{j=1}^n -\Delta_{x_j}\right)^{1/2}.
  $$

   Since $h_0$ is not bounded, we proceed by regularization in order to use \fref{def1} for the definition of the one-body density matrix. For $\eps>0$,  set then $h_\eps=h_0(\mathrm{Id}_{\mathfrak{h}}+\eps h_0)^{-1} \in \calL(\frakh)$. According to \fref{def1}, we have
  \be \label{rel}
  \mathrm{Tr}_\frakh \left( h_\eps^{1/2} \varrho^{(1)} h_\eps^{1/2} \right)=\mathrm{Tr}_\frakh \left( h_\eps \varrho^{(1)}  \right)=
  \Tr\left( d\Gamma(h_\eps) \varrho \right).
  \ee
  The last term is equal to
  \be \label{rel1}
  \sum_{p \in \NN} \sum_{n \in \NN^*} \rho_p \left\| \AA^\eps_{(n)} \psi_p^{(n)}\right\|^2_{n},
  \ee
  with
  $$
\AA^\eps_{(n)}=\left(\sum_{j=1}^n -\Delta_{x_j} (\mathrm{Id}_{\mathfrak{h}}- \eps \Delta_{x_j})^{-1}\right)^{1/2}.
  $$
  For $\psi \in \mathfrak{F}_{b/f}^{(n)}$, we find by a Fourier transform
  \bea  \label{rel2}
  \| \AA^\eps_{(n)} \psi \|^2_{(n)} =&(2 \pi)^{-n} \int_{(\Rm^d)^n} \left(\sum_{j=1}^n \frac{|k_j|^2}{1+\eps |k_j|^2}\right)^{1/2} |\widehat \psi(k_1, \cdots,k_d)|^2 d k_1 \cdots d k_d\\\nonumber
  \leq&  \| d\Gamma(h_0)^{1/2}_{(n)} \psi \|^2_{(n)} .
  \eea
Hence, 
$$
\mathrm{Tr}_\frakh \left( h_\eps^{1/2} \varrho^{(1)} h_\eps^{1/2} \right) \leq \Tr\left( \HH_0^{1/2} \varrho \,\HH_0^{1/2} \right),
$$  
and there exists therefore $\alpha \in \calJ_1(\frakh)$ and a subsequence such that $h_{\eps_\ell}^{1/2} \varrho^{(1)} h_{\eps_\ell}^{1/2} \to \alpha$ in  $\calJ_1(\frakh)$ weak-$*$ as $\ell \to \infty $ with
   $$ \mathrm{Tr}_\frakh \left( \alpha \right) \leq \Tr\left( \HH_0^{1/2} \varrho \,\HH_0^{1/2} \right).$$
We now identify $\alpha$ with $h_0^{1/2} \varrho^{(1)} h_0^{1/2}$. Let $K$ be a compact operator on $\frakh$. Then,
$$
   \mathrm{Tr}_\frakh \left( (\mathrm{Id}_{\mathfrak{h}}+h_0)^{-1} K (\mathrm{Id}_{\mathfrak{h}}+h_0)^{-1} h_\eps^{1/2} \varrho^{(1)} h_\eps^{1/2} \right)= \mathrm{Tr}_\frakh (K_\eps \varrho^{(1)}),
   $$
   with
   $$
   K_\eps=h_\eps^{1/2} (\mathrm{Id}_{\mathfrak{h}}+h_0)^{-1} K (\mathrm{Id}_{\mathfrak{h}}+h_0)^{-1} h_\eps^{1/2}.
   $$
   As $\eps \to 0$, the operator $K_\eps$ converges strongly to $h_0^{1/2} (\mathrm{Id}_{\mathfrak{h}}+h_0)^{-1} K (\mathrm{Id}_{\mathfrak{h}}+h_0)^{-1} h_0^{1/2}$ in $\calL(\frakh)$. As a consequence
   \begin{align*}
   \mathrm{Tr}_\frakh \big( (\mathrm{Id}_{\mathfrak{h}}+h_0)^{-1} K& (\mathrm{Id}_{\mathfrak{h}}+h_0)^{-1} \alpha \big)\\
&= \mathrm{Tr}_\frakh \left( K h_0^{1/2} (\mathrm{Id}_{\mathfrak{h}}+h_0)^{-1} \varrho^{(1)} (\mathrm{Id}_{\mathfrak{h}}+h_0)^{-1} h_0^{1/2} \right),
   \end{align*}
   which allows us to identify $\alpha$  with $h_0^{1/2} \varrho^{(1)} h_0^{1/2}$. The relation \fref{eqkin} is obtained by passing to the limit in \fref{rel}: in the l.h.s, we use the fact that $h_0^{1/2} \varrho^{(1)} h_0^{1/2} \in \calJ_1(\frakh)$, and in the r.h.s., we use \fref{rel1}, \fref{rel2}, and monotone convergence. This ends the proof.

\bibliographystyle{plain}
\bibliography{bibliography.bib}

\end{document}